\documentclass[reqno,12pt]{amsart}

\usepackage{amssymb,latexsym,amscd} 
\usepackage{amsmath,geometry}
\makeatletter
\renewcommand{\pod}[1]{\allowbreak\mathchoice
  {\if@display \mkern 18mu\else \mkern 8mu\fi (#1)}
  {\if@display \mkern 18mu\else \mkern 8mu\fi (#1)}
  {\mkern4mu(#1)}
  {\mkern4mu(#1)}
}
\usepackage{enumerate}[1.)]
\usepackage{pb-diagram}  
\renewcommand{\eqref}[1]{(\ref{#1})}   
\theoremstyle{plain}
\newtheorem{theorem}{Theorem}[section]

\newtheorem{corollary}{Corollary}[section]

\newtheorem{lemma}{Lemma}[section]

\newtheorem{algorithm}{Algorithm}

\newtheorem{definition}{Definition}

\begin{document}

\title{Another application  of Linnik's dispersion method}
  
\date{\today}
\author{\'Etienne Fouvry}
\address{Laboratoire de Math\' ematiques d'Orsay, Univ. Paris--Sud, CNRS, Universit\' e Paris--Saclay,    91405 Orsay, France}
\email{Etienne.Fouvry@u-psud.fr}
\author{Maksym Radziwi\l\l}
\address{
  Department of Mathematics \\
Caltech \\
1200 E California Blvd \\
Pasadena \\ CA \\ 91125
  }

%
\email{maksym.radziwill@gmail.com}
 
\dedicatory{In memoriam Professor Yu. V. Linnik (1915-1972)}

\begin{abstract} Let $\alpha_m$ and $\beta_n$ be two sequences of real numbers supported on $[M, 2M]$ and $[N, 2N]$ with $M = X^{1/2 - \delta}$ and $N = X^{1/2 + \delta}$. We show that there exists a $\delta_0 > 0$ such that the multiplicative convolution of $\alpha_m$ and $\beta_n$ has exponent of distribution $\frac{1}{2} + \delta-\varepsilon$ (in a weak sense) as long as $0 \leq \delta < \delta_0$,    the sequence  $\beta_n$ is Siegel-Walfisz and both sequences $\alpha_m$ and $\beta_n$ are bounded above by divisor functions. Our result is thus a general dispersion estimate for ``narrow'' type-II sums. The proof relies crucially on Linnik's dispersion method and recent bounds for trilinear forms in Kloosterman fractions due to Bettin-Chandee. We highlight an application related to the Titchmarsh divisor problem. 
\end{abstract}

\maketitle

\renewcommand{\theenumi}{(\roman{enumi})}
 

\section{Introduction} An important theme in analytic number theory is the study of the distribution of sequences in arithmetic progressions. A representative result in this field is the Bombieri-Vinogradov theorem \cite{BV}, according to which for any $A > 0$, 
\begin{equation} \label{eq:primes} 
  \sum_{q \leq Q} \max_{(a,q) = 1} \Big | \sum_{\substack{p \leq x \\ p \equiv a \pmod{q}}} 1 - \frac{1}{\varphi(q)} \sum_{p \leq x} 1 \Big | \ll_{A} x (\log x)^{-A}
\end{equation}
provided that $Q \leq \sqrt{x} (\log x)^{-B}$ for some constant $B = B(A)$ depending on $A > 0$.

Nothing of the strength of \eqref{eq:primes} is known in the range $Q > x^{1/2 + \varepsilon}$ for any fixed $\varepsilon > 0$ and already establishing for any fixed integer $a \neq 0$ and for all $A > 0$ the weaker estimate,
\begin{equation} \label{eq:primes2}
\sum_{q \leq Q} \Big | \sum_{\substack{p \leq x \\ p \equiv a \pmod{q}}} 1 - \frac{1}{\varphi(q)} \sum_{p \leq x} 1 \Big | \ll_{a,A} x (\log x)^{-A}
\end{equation}
with $Q = x^{1/2 + \delta}$ and some $\delta > 0$ is a major open problem. If we could show \eqref{eq:primes2} then we would say that \textit{the primes have exponent of distribution $\frac 12 + \delta$ in a weak sense}. However we note that there are results of this type if one allows to restrict the sum over $q \leq Q$ in \eqref{eq:primes2} to integers that are $x^{\varepsilon}$ smooth, for a sufficiently small $\varepsilon > 0$ (see \cite{Zhang, Polymath}).

Any known approach to \eqref{eq:primes2} goes through combinatorial formulas which decompose the sequence of prime numbers as a linear combination of multiplicative convolutions of other sequences (see for example \cite[Chapter 13]{I-K}). 
If one attempts to establish \eqref{eq:primes2} by using such a combinatorial formula then one is led to the problem of showing that for any $A > 0$, 
  \begin{equation} \label{eq:geh}
  \sum_{\substack{q \leq Q \\ (q,a) = 1}} \Big | \sum_{\substack{M \leq m \leq 2M \\ N \leq n \leq 2N \\ m n \equiv a \pmod{q}}} \alpha_m \beta_n - \frac{1}{\varphi(q)} \sum_{\substack{M \leq m \leq 2M \\ N \leq n \leq 2N \\ (m n, q) = 1}} \alpha_{m} \beta_{n} \Big | \ll X (\log X)^{-A} \ , \ X := M N
  \end{equation}
  with $Q > X^{1/2 + \varepsilon}$ for some $\varepsilon > 0$. In \cite{Li} Linnik developed his ``dispersion method'' to tackle such expressions. The method relies crucially on the bilinearity of the problem, followed by the use of various estimates for Kloosterman sums of analytic or algebraic origins. 
 For a bound such as \eqref{eq:geh} to hold one needs to impose a ``Siegel-Walfisz condition'' on at least one of the sequences $\alpha_m$ or $\beta_n$. 

  \begin{definition} 
  We say that a sequence $\boldsymbol \beta = (\beta_n)$ satisfies a Siegel-Walfisz condition (alternatively we also say that $\boldsymbol \beta$ is \textit{Siegel-Walfisz}), if there exists an integer $k > 0$ such that for any fixed $A > 0$, uniformly in $x \geq 2$, $q > |a| \geq 1, r \geq 1$ and $(a,q) = 1$, we have,
  \begin{equation*} \label{condSW}
  \sum_{\substack{x < n \leq 2x \\ n \equiv a \pmod{q} \\ (n,r) = 1}} \beta_n - \frac{1}{\varphi(q)} \sum_{\substack{x < n \leq 2x \\ (n, q r) = 1}} \beta_{n}  = O_{A}(\tau_k(r) \cdot x (\log x)^{-A}). 
\end{equation*}
where $\tau_k(n) := \sum_{n_1 \ldots n_k = n} 1$ is the $k$th divisor function. 
\end{definition}

It is widely expected (see e.g \cite[Conjecture 1]{BFI1}) that \eqref{eq:geh} should hold as soon as $\min(M,N) > X^{\varepsilon}$ provided that at least one of the sequences $\alpha_n,\beta_m$ is Siegel-Walfisz, and that there exists an integer $k > 0$ such that $|\alpha_m| \leq \tau_k(m)$ and $|\beta_n| \leq \tau_k(n)$ for all integers $m,\, n \geq 1$. We are however very far from proving a result of this type.

When $Q > X^{1/2 + \varepsilon}$ for some $\varepsilon > 0$, there are only a few results establishing \eqref{eq:geh} unconditionally in specific ranges of $M$ and $N$ (precisely \cite[Th\' eor\`eme 1]{FoActaMath}, \cite[Theorem 3]{BFI1}, \cite[Corollaire 1]{FoAnnENS}, \cite[Corollary 1.1 (i)]{FouRadz2}). All the results that establish \eqref{eq:geh} unconditionally place a restriction on one of the variable $N$ or $M$ being much smaller than the other. We call such cases ``unbalanced convolutions'' and this forms the topic of our previous paper \cite{FouRadz2}. 

In applications a recurring range is one where $M$ and $N$ are roughly of the same size. This often corresponds to the case of ``type II sums'' in which one is permitted to exploit bilinearity but not much else. This is the range to which we contribute in this paper. 

\begin{theorem} \label{thm:main}
  Let $k \geq 1$ be an integer and $M,N \geq 1$ be given. Set $X = M N$. Let $\alpha_m$ and $\beta_n$ be two sequences of real numbers supported respectively on $[M, 2M]$ and $[N, 2N]$. Suppose that $\boldsymbol \beta = (\beta_n)$ is Siegel--Walfisz and suppose that $|\alpha_{m}| \leq \tau_k(m)$ and $|\beta_n| \leq \tau_k(n)$ for all integers $m,n \geq 1$. Then, for every $\varepsilon > 0$ and every $A > 0$,
  \begin{equation} \label{eq:maineq}
  \sum_{\substack{Q \leq q \leq 2Q \\ (q,a) = 1}} \Big | \sum_{\substack{m n \equiv a \pmod{q}}} \alpha_m \beta_n - \frac{1}{\varphi(q)} \sum_{(m n, q) = 1} \alpha_m \beta_n \Big | \ll_{A} X (\log X)^{-A} 
  \end{equation}
  uniformly in $N^{56/23} X^{-17/23 + \varepsilon} \leq Q \leq N X^{-\varepsilon}$ and $1 \leq |a| \leq X$.
\end{theorem}


Setting $N = X^{1/2 + \delta}$ and $M = X^{1/2 - \delta}$ in Theorem \ref{thm:main} it follows from Theorem \ref{thm:main} and the Bombieri-Vinogradov theorem that \eqref{eq:maineq} holds  for
all $Q \leq NX^{-\varepsilon}$ with $0\leq \delta < \delta_0:=\frac{1}{112}.$
  Previously the existence of such a $\delta_0 > 0$ was established conditionally on Hooley's $R^{\star}$ conjecture on cancellations in short incomplete Kloosterman sums in \cite[Th\' eor\`eme 1]{FoActaArith} and in that case one can take $\delta_0 = \frac{1}{14}$. Similarly to our previous paper, we use the work of Bettin-Chandee \cite{Be-Ch} and Duke-Friedlander-Iwaniec \cite{D-F-I} as an unconditional substitute for Hooley's $R^{\star}$ conjecture. In fact the proof of Theorem \ref{thm:main} follows closely the proof of the conditional result in \cite[Th\'eor\`eme 1]{FoActaArith} up to the point where Hooley's $R^{\star}$ conjecture is applied. Incidentally we notice that the largest $Q$ that Theorem \ref{thm:main} allows to take is $Q = X^{17/33 - 5 \varepsilon}$ provided that one chooses $N = X^{17/33 - 4 \varepsilon}$. 

Unfortunately the type-II sums that our Theorem \ref{thm:main} allows to estimate are too narrow to make Theorem \ref{thm:main} widely applicable in many problems (however see \cite{Tao} for an interesting connection with cancellations in character sums). We record nonetheless below one corollary, which is related to Titchmarsh's divisor problem concerning the estimation of $\sum_{p \leq x} \tau_2(p - 1)$ (for the best results on this problem see \cite[Corollaire  2]{FoCrelle}, \cite[Corollary 1]{BFI1} and \cite{Drappeau}). The proof of the Corollary below will be given in \S \ref{proofcorollary}. 

    \begin{corollary}\label{appli} Let $k \geq 1$ and 
  let $\boldsymbol \alpha$ and $\boldsymbol \beta$ be two sequences of real numbers as in  Theorem  \ref{thm:main}. Let $\delta$ be a constant satisfying
  $$
  0 < \delta < \frac{1}{112},
  $$
  and let $$ X\geq 2,\ M=X^{1/2-\delta}, \text{ and } N=X^{1/2 + \delta}.
  $$
 Then  for every $A >0$ we have the equality
\begin{equation*} 
  \sum_{m\sim M} \sum_{n \sim N} \alpha_m \beta_n \tau_2 (mn-1)
  = 2 \sum_{q\geq 1 } \frac{1}{\varphi (q)} \underset{\substack{m\sim M, n\sim N \\ mn >q^2 \\ (mn,q) =1}} {\sum \sum} \alpha_m \beta_n
  +
  O \bigl( X (\log X)^{-A}\bigr).
   \end{equation*}
  \end{corollary}

      

\section{Conventions and lemmas}

\subsection{Conventions} For  $M$ and $N\geq 1$, we put $X=MN$ and $\mathcal L =\log 2X$.
Whenever it appears in the subscript of a sum the notation $n \sim N$ will means $N \leq n < 2N$. Given an integer $a \not= 0$ and  two sequences  $\boldsymbol \alpha = (\alpha_m)_{M \leq m < 2M}$ and $\boldsymbol \beta = (\beta_n)_{N \leq n < 2N}$ supported respectively on $[M, 2M]$ and $[N, 2N]$ we define the discrepancy

\begin{equation*}
E (\boldsymbol \alpha, \boldsymbol \beta, M, N, q,a) :=
\underset{\substack{m\sim M, n\sim N \\ mn\equiv a \bmod q}}{\sum \sum} \alpha_m \beta_n -\frac{1}{\varphi (q)}
 \underset{\substack{m\sim M, n\sim N \\ (mn, q)=1}}{\sum \sum} \alpha_m \beta_n,  
\end{equation*}
and we also define the mean-discrepancy,
\begin{equation} \label{defDelta}
\Delta (\boldsymbol \alpha, \boldsymbol \beta, M, N, q, a) := \sum_{\substack{q \sim Q \\ (q,a)=1}} |E ( \boldsymbol \alpha, \boldsymbol \beta, M, N, q, a)|.
\end{equation}

Throughout $\eta$ will denote any positive number
the value of which may change at each occurence. The dependency on $\eta$ will not be recalled in the $O$ or $\ll$--symbols. Typical examples are  $\tau_k (n) = O(n^\eta)$ or $(\log x)^{10} =O (x^\eta)$, uniformly for $x\geq 1$.

If $f$ is a smooth real function, its Fourier transform is defined by
$$
\hat f (\xi) = \int_{-\infty}^\infty f(t) e( -\xi t) \, {\rm d} t,
$$
where $e(\cdot)= \exp (2 \pi i \cdot).$

\subsection{Lemmas}
Our first lemma is  a  classical finite version of  the Poisson summation formula 
in arithmetic progressions, with a good error term.

\begin{lemma} \label{existenceofpsi}There exists  
 a  smooth function $\psi\ : \ \mathbb R \longrightarrow \mathbb R^+$, with compact support equal to $[1/2, 5/2]$,  larger than the characteristic function of the interval $[1,2]$, equal to $1$ on this interval 
such that,   uniformly for  integers $a$ and $q \geq 1$,  for  $M \geq 1$ and $H\geq  (q/M)\log ^4 2M$
one has the equality
\begin{equation}\label{ineqpsi}
\sum_{m\equiv a \bmod q} \psi \Bigl( \frac{m}{M}\Bigr) =\hat{\psi}(0)  \frac {M}{q}
+ \frac{M}{q} \sum_{0 < \vert h \vert \leq H}e \bigl( \frac{ ah}{q} \bigr)\hat \psi \Bigl( \frac{h}{q/M}\Bigr) +O(M^{-1}).
\end{equation}
Furthermore, uniformly for $q\geq 1$ and $M\geq 1$ one has the equality
\begin{equation}\label{Poissoncoprime}
\sum_{(m,q)=1} \psi \Bigl( \frac{m}{M}\Bigr) =\frac{\varphi (q)}{q}\hat{\psi}(0)  M + O \bigl(\tau_2 (q) \log^4 2M \bigr).
\end{equation}

\end{lemma}
\begin{proof}  See Lemma 2.1 of \cite{FouRadz2}, inspired by \cite[Lemma 7]{B-F-I2}.
\end{proof}
We now recall a classical lemma on the average behavior of the $\tau_k$-function in arithmetic progressions (see \cite[Lemma 1.1.5]{Li}, for instance).
\begin{lemma}\label{dkinarith} 
 For every $k\geq 1$,  for every $\varepsilon >0$, there exists $C(k,\varepsilon)$ such that, for every $x \geq 2$,  for every $x^\varepsilon < y < x$, for every $1\leq q \leq  yx^{ -\varepsilon}$, for every  integer $a$ coprime with $q$, one has the inequality $$
\sum_{\substack{x-y <  n \leq x \\ n \equiv a \bmod q}} \tau_k (n) \leq C({k,  \varepsilon})\frac{ y }{\varphi (q)}(\log 2x)^{k-1}. 
$$
\end{lemma}

The following lemma is one of the various forms of the so--called Barban--Davenport--Halberstam Theorem (for a proof see for instance
\cite[Theorem 0 (a)]{BFI1}.
\begin{lemma}\label{Ba-Da-Ha}Let $k > 0$ be an integer. Let $\boldsymbol \beta = (\beta_{n})$ be a  Siegel--Walfisz sequence such that $|\beta_{n}| \leq \tau_k(n)$ for all integer $n \geq 1$.   Then for every $A > 0$ there exists $B=B(A)$ such that, uniformly for $N \geq 1$ one has the equality
$$
\sum_{q\leq N(\log 2N)^{-B}}\ \sum_{a,\, (a,q) =1} \Bigl\vert \sum_{\substack{n\sim N \\ n\equiv a  \bmod q}} \beta_n -\frac{1}{\varphi (q)} \sum_{\substack{n\sim N \\ (n,q) =1}} \beta_n
\Bigr\vert^2 =O_A \bigl( N (\log 2N)^{-A}\bigr).
$$
\end{lemma}
We now recall an easy consequence of Weil's bound for Kloosterman sums.
\begin{lemma}\label{shortkloo}
Let $a$ and $b$ two integers $\geq 1$. Let $\mathcal I$ an interval included in $[1, a]$. Then for every integer $\ell$ for every $\varepsilon >0$  we  have the inequality
$$
\sum_{\substack{n\in \mathcal I \\ (n,ab) =1}} \frac{n}{\varphi (n)} e\Bigl(  \ell \frac{\overline n}{a}\Bigr) = O_\varepsilon \Bigl( (\ell, a)^\frac 12 (ab)^\varepsilon a^\frac12
\Bigr).
$$
\end{lemma}
\begin{proof} We begin we the case $b=1$. We write the factor $\frac{n}{\varphi (n)} $ as
$$
\frac{n}{\varphi (n)} = \sum_{\nu \mid n^\infty} \nu^{-1} =  \sum_{\kappa (\nu) \mid n } \nu^{-1},
$$
where $\kappa (\nu)$   is the largest   squarefree integer dividing $\nu$ (sometimes $\kappa (\nu)$ is called the {\it kernel } of $\nu$). This gives the equality
$$
\Bigl\vert \ \sum_{\substack{n\in \mathcal I \\ (n,a) = 1}} \frac{n}{\varphi (n)} e\Bigl(  \ell \frac{\overline n}{a}\Bigr) \ \Bigr\vert \leq 
\sum_{\nu \geq 1} \nu^{-1} \Bigl\vert \ \sum_{\substack{n\in \mathcal I  \\ \kappa (\nu) \mid  n  \\ (n,a) = 1}} 
e\Bigl(  \ell \frac{\overline n}{a}\Bigr) \ \Bigr\vert 
 = \sum_{\substack{\nu \geq 1 \\ (\nu, a)=1}} \nu^{-1} \Bigl\vert \ \sum_{\substack{m\in \mathcal I /\kappa (\nu)  \\ (m,a) = 1}} 
e\Bigl(  \ell \frac{\overline {\kappa (\nu)}\, \overline m}{a}\Bigr) \ \Bigr\vert. 
$$
In the summation we can restrict to the $\nu$ such that $\kappa (\nu) \leq a$. Applying the classical bound for short Kloosterman sums, we deduce
that 
$$
\Bigl\vert \ \sum_{\substack{n\in \mathcal I \\ (n,a) = 1}} \frac{n}{\varphi (n)} e\Bigl(  \ell \frac{\overline n}{a}\Bigr) \ \Bigr\vert \ll_\varepsilon (\ell, a)^\frac{1}{2} a^{\frac 12 + \varepsilon} \prod_{p \leq a} \Bigl(1-\frac{1}{p}\Bigr)^{-1} \ll_\varepsilon   (\ell, a)^\frac{1}{2} a^{\frac 12 + 2 \varepsilon}.
$$
This proves Lemma \ref{shortkloo} in the case where $b=1$. When $b \not=1$, we use the M\" obius inversion formula to detect the condition $(n,b)=1$.
\end{proof}

Our central tool is a bound for trilinear forms for Kloosterman fractions, due to Bettin and Chandee  \cite[Theorem 1]{Be-Ch}. The result of Bettin-Chandee builds on work of Duke-Friedlander-Iwaniec \cite[Theorem 2]{D-F-I} who considered the case of bilinear forms.
These two papers show cancellations in  exponential sums involving Kloosterman fractions $e(a \overline{m} / n)$ with $m \asymp n$. We state below the main theorem of Bettin-Chandee.

 \begin{lemma}\label{trilinear} For every $\epsilon >0$ there exists $C(\varepsilon)$ such that
 for every non zero integer $\vartheta$, for every sequences  let $\boldsymbol \alpha$, $\boldsymbol \beta$ and $\boldsymbol \nu$ be   of complex numbers, for every $A$, $M$ and $N\geq 1$, one has the inequality
 \begin{multline*}
\Bigl\vert \, \sum_{a\sim A} \sum_{m\sim M} \sum_{n \sim N} \alpha (m) \beta (n) \nu (a) e \Bigl(\vartheta \frac{a \overline m}{n} \Bigr)\,   \Bigr\vert \leq C(\varepsilon) \Vert \boldsymbol \alpha\Vert_{2, M} \, \Vert \boldsymbol \beta\Vert_{2, N} \, \Vert \boldsymbol \nu\Vert_{2, A} 
\\
\times  \Bigl( 1+ \frac{\vert \vartheta\vert A}{MN}\Bigr)^\frac{1}{2}  \Bigl( (AMN)^{\frac{7}{20} + \varepsilon} \, (M+N) ^\frac{1}{4} +(AMN)^{\frac{3}{8} +\varepsilon} (AM+AN) ^\frac{1}{8}
\Bigr). 
 \end{multline*}
 \end{lemma}

\section{Proof of Theorem \ref{thm:main}} 
All along the proof we will suppose that the inequality $1 \leq |a| \leq X$ holds and  that we also have 
\begin{equation}\label{universal}
X ^\frac 38 \leq M\leq X^\frac 12  \leq N \text{ and }         Q\leq N.
\end{equation}
\subsection{Beginning of the dispersion} Without loss of generality we can suppose that the sequence   $\boldsymbol \beta$ satisfies the following property
\begin{equation}\label{betan=0}
n\mid a \Rightarrow \beta_n= 0.
\end{equation}
Such an assumption is justified because the contribution to $ \Delta( \boldsymbol \alpha, \boldsymbol \beta, M, N, Q,a) $ of the $(q,m,n)$ 
such that $n \mid a$ is 
$$
\ll  QX^\eta+X^\eta  \sum_{n \mid a}\sum_{\substack{m\sim M \\ mn\not= a}} \tau_2(\vert mn -a\vert) + M X^\varepsilon \ll (M+Q)X^\eta.
$$

By \eqref{defDelta}, we have the inequality 
\begin{equation*} 
 \Delta( \boldsymbol \alpha, \boldsymbol \beta, M, N, Q,a) \leq \sum_{q\sim Q} \, \sum_{\substack{m\sim M \\ (m,q)=1}} \vert \alpha_m \vert 
 \Bigl\vert  \sum_{\substack{n \sim N\\ n\equiv a \overline m\bmod q}} \beta_n - \frac{1}{\varphi (q)} \sum_{\substack{n\sim N \\ (n,q)=1}} \beta_n \Bigr\vert.
\end{equation*}
Let $\psi$ be the smooth function constructed in  Lemma \ref{existenceofpsi}.
 By  the  Cauchy--Schwarz inequality, the inequality $|\alpha_m| \leq \tau_k(m)$ and by Lemma \ref{dkinarith}  we deduce
\begin{equation}\label{CS}
  \Delta^2( \boldsymbol \alpha, \boldsymbol \beta, M, N, Q,a)  \ll MQ \mathcal L^{k^2-1}\, \Bigl\{ W(Q) -2 V(Q) +U(Q)\Bigr\},
  \end{equation}
  with
  \begin{align}
  U(Q)&= \sum_{  (q,a)=1} \frac{\psi (q/Q)}{\varphi^2 (q)}\,\Bigl( \sum_{\substack{n \sim N \\ (n,q)=1}} \beta_n \Bigr)^2 \sum_{(m,q)=1} \psi \Bigl( \frac{m}{M} \Bigr),\label{defU}\\
  V(Q)& = \sum_{ (q,a)=1} \frac{\psi (q/Q)}{\varphi (q)}\,
  \Bigl( \sum_{\substack{n_1 \sim N \\ (n_1,q)=1}} \beta_{n_1} \Bigr)
  \Bigl( \sum_{\substack{n_2 \sim N \\ (n_2,q)=1}} \beta_{n_2} \Bigr) \sum_{  m\equiv a \overline{n_1}\bmod q } \psi \Bigl( \frac{m}{M} \Bigr),\nonumber
  \\
  W(Q) & =  \sum_{(q,a)=1} \psi (q/Q) 
  \Bigl( \sum_{\substack{n_1 \sim N\\ (n_1,q)=1}} \beta_{n_1} \Bigr)
  \Bigl( \sum_{\substack{n_2 \sim N\\ (n_2,q)=1}} \beta_{n_2} \Bigr) \sum_{\substack{ m\equiv a \overline{n_1}\bmod q \\ m\equiv a \overline{n_2}\bmod q}} \psi \Bigl( \frac{m}{M} \Bigr).\label{defW}
 \end{align}

  \subsection{Study of $U(Q)$} A direct application of \eqref{Poissoncoprime} of Lemma  \ref{existenceofpsi}  in the definition \eqref{defU} gives the equality
  \begin{align}
  U(Q) &= \hat \psi (0) M  \sum_{ (q,a)=1} \frac{\psi (q/Q)}{q\varphi (q)}\,\Bigl( \sum_{\substack{n \sim N\\ (n,q)=1}} \beta_n \Bigr)^2 + O \bigl(  N^2 Q^{-1} X^\eta
  \bigr)\nonumber\\
  &= U^{\rm MT} (Q) +  O \bigl(  N^2 Q^{-1}X^\eta \label{UMT}
  \bigr),
  \end{align}
  by definition.
  \subsection{Study of $V(Q)$}  Let $\varepsilon$ be a fixed positive number. We now apply \eqref{ineqpsi} of Lemma \ref{existenceofpsi}   with 
  \begin{equation}\label{defH}
  H = M^{-1} Q X^\varepsilon.
  \end{equation}
   This  leads to the equality
\begin{equation}\label{V1}
  V(Q)= V^{\rm MT}(Q) + V^{\rm Err1} (Q) + V^{\rm Err2} (Q),
  \end{equation}
  where each of the three terms corresponds to the contribution of the three terms on the right hand--side of \eqref{ineqpsi}. We directly have the equality
  \begin{equation}\label{V2}
  V^{\rm Err2} (Q) = O \bigl(   M^{-1} N^2 X^\eta \bigr).
 \end{equation}
  For  the main term we get
\begin{equation}\label{V3}
  V^{\rm MT} (Q) = \hat \psi (0) M  \sum_{ (q,a)=1} \frac{\psi (q/Q)}{q\varphi (q)}\,\Bigl( \sum_{\substack{n \sim N\\ (n,q)=1}} \beta_n \Bigr)^2.
 \end{equation}
  By the definition of $V^{\rm Err1} (Q)$ we have the equality
  \begin{multline*}
  V^{\rm Err1} (Q) = M \sum_{ (q,a)=1} \frac{\psi (q/Q)}{q\varphi (q)}\,
  \Bigl( \sum_{\substack{n_2 \sim N\\ (n_2,q)=1}} \beta_{n_2} \Bigr)\\
  \Bigl( \sum_{\substack{n_1 \sim N\\ (n_1,q)=1}} \beta_{n_1} 
   \sum_{0 < \vert h \vert \leq H} \hat \psi \Bigl(\frac{h}{q/M}\Bigr) e \Bigl( \frac{ah \overline{n_1}}{q} \Bigr)\Bigr),
  \end{multline*}
  from which we deduce the inequality
\begin{equation}\label{split1}
  \bigl\vert \, V^{\rm Err1} (Q) \, \bigr\vert \leq   M   Q^{-2} \sum_{n_1\sim N} \vert  \beta_{n_1}\vert \sum_{n_2\sim N} \vert  \beta_{n_2}\vert 
  \sum_{0< \vert h \vert \leq H} \bigl\vert \,\mathcal V (n_1, n_2, h) \, \bigr\vert 
\end{equation}
  with 
  $$
  \mathcal V (n_1, n_2,  h) = \sum_{ (q,an_1n_2)=1} \psi (q/Q)\frac{Q^2}{q \varphi (q)} \hat \psi \Bigl(\frac{h}{q/M}\Bigr) e \Bigl( \frac{ah \overline{n_1}}{q} \Bigr).
  $$
  Since $(q,n_1)=1$  B\' ezout's relation gives the equality  $$
  \frac{ah \overline{n_1}}{q} = -ah \frac{\overline q}{n_1} +  \frac{ah}{n_1 q} \mod 1.
  $$
 By the inequality $1 \leq |a| \leq X$  and by the definition of $H$, the derivative of the  bounded function
 $$t\mapsto \psi (t/Q)\frac{Q^2}{t^2}\hat \psi \Bigl(\frac{h}{t/M}\Bigr) e\Bigl( \frac {ah}{n_1 t}\Bigr)$$
  is  $\ll X^\varepsilon t^{-1}$ when $t \asymp Q$.    This allows to make a partial summation over the variable $q$ with the loss
  of a factor $X^\varepsilon$. After all these considerations, we see that there exists a subinterval $\mathcal J \subset [Q/2, 5 Q/2]$ such that we have the inequality 
  $$
 \bigl\vert \,   \mathcal V (n_1, n_2,  h)\, \bigr\vert  \ll  X^\varepsilon \,\Bigl\vert \sum_{\substack{q \in \mathcal J\\ (q, n_1n_2)=1}} \frac{q}{\varphi (q)} e \Bigl(
 ah \frac{\overline q}{ n_1}\Bigr)\, \Bigr\vert.
 $$
   Lemma \ref{shortkloo} leads to the bound
  $$
   \bigl\vert \,   \mathcal V (n_1, n_2,  h)\, \bigr\vert  \ll  X^\varepsilon (ah, n_1)^\frac12 (n_1n_2)^\eta  n_1^\frac 12.
   $$
  Inserting this into \eqref{split1}, we obtain 
  $$
  V^{\rm Err1} (Q) \ll MN^\frac 32 Q^{-2} X^{\varepsilon +\eta}  \sum_{n_1 \sim N} \vert \beta_{n_1}\vert  \sum_{0< \vert h \vert \leq H} (h,n_1)^\frac 12,
  $$
  which finally gives
\begin{equation}\label{V4}
  V^{\rm Err1} (Q) \ll  N^\frac 52 Q^{-1} X^{2\varepsilon +\eta}
\end{equation}
 using the inequality $|\beta_n| \leq \tau_k(n)$ and the definition of $H$. Combining \eqref{V1}, \eqref{V2}, \eqref{V3} and \eqref{V4} we obtain the equality
 \begin{equation}\label{V5}
 V(Q) = V^{\rm MT} (Q) + O_\varepsilon \bigl( (M^{-1} N^2 +N ^\frac 52 Q^{-1})X^{2\varepsilon+\eta} \bigr).
 \end{equation}
  where $V^{\rm MT} (Q)$ is defined in \eqref{V3} and where the constant implicit in the $O_\varepsilon$--symbol is uniform for 
 $a$ satisfying $1 \leq |a| \leq X$. 
 \section{Study of $W(Q)$}
  \subsection{ The preparation of the variables } The conditions of the last summation in \eqref{defW} imply the congruence restriction
  \begin{equation}\label{n1=n2}
  n_1\equiv n_2 \bmod q \text{ and } (n_1n_2, q)=1.
  \end{equation}
  In order to control the mutual multiplicative properties of $n_1$ and $n_2$ we  decompose these variables as
  \begin{equation}\label{decompn1n2}
  \begin{cases} (n_1, n_2 ) =d,\\
  n_1= d \nu_1, \ n_2 = d \nu_2, \  (\nu_1, \nu_2)=1, \\
  \nu_1= d_1\nu'_1\text { with } d_1 \mid d^\infty\text{ and } (\nu'_1, d)=1.
  \end{cases}
  \end{equation} 
  Thanks to $|\beta_n| \leq \tau_k(n)$ and to \eqref{betan=0}  the contribution of the pairs $(n_1,n_2)$ with $d > X^\varepsilon$ to  the right--hand side of \eqref{defW}  is negligible since it is
  \begin{align}
 & \ll X^\eta \sum_{ X^\varepsilon <d \leq 2N} \sum_{m\sim M} \, \sum_{\substack{\nu_1 \sim N/d \\ dm\nu_1-a \not=0}} \sum_{\substack{q\sim Q \\ q \mid d\nu_1 m-a}} 
 \sum_{\substack{\nu_2 \sim N/d \\ \nu_2\equiv  \nu_1 \bmod q}} 1  \nonumber \\
 &  \ll  X^\eta  \sum_{ X^\varepsilon <d \leq 2N} \sum_{m\sim M} \, \sum_{\substack{\nu_1 \sim N/d\\ dm\nu_1-a \not=0}} \tau_2 (\vert d\nu_1 m-a\vert) \Bigl( \frac{ N}{d Q} +1\Bigr) \nonumber\\
 & \ll  MN^2 Q^{-1} X^{\eta -\varepsilon} +  X^{1 + \eta}. \label{d>X}
  \end{align}
  Now consider the contribution  of the pairs $(n_1,n_2)$ with $d \leq  X^\varepsilon$ and $d_1 >X^\varepsilon$ to  the right--hand side of \eqref{defW}. It   is
  \begin{align}
 & \ll   X^\eta \sum_{ d \leq X^\varepsilon} \sum_{\substack{X^\varepsilon <d_1 <2N\\ d_1 \mid d^\infty}} \sum_{m\sim M} \, \sum_{\substack{\nu'_1 \sim N/(dd_1) \\ dd_1m\nu'_1-a \not=0}} \sum_{\substack{q\sim Q \\ q \mid dd_1\nu'_1 m-a}} 
 \sum_{\substack{\nu_2 \sim N/d \\ \nu_2\equiv d_1 \nu'_1 \bmod q}} 1  \nonumber \\
 & \ll   X^\eta  \sum_{ d \leq X^\varepsilon} \sum_{\substack{X^\varepsilon <d_1 <2N\\ d_1 \mid d^\infty}} \sum_{m\sim M} \, \sum_{\substack{\nu'_1 \sim N/(d d_1)\\ dd_1m\nu'_1-a \not=0}} \tau_2 (\vert dd_1\nu'_1 m-a\vert) 
 \Bigl( \frac{ N}{d Q} +1\Bigr) \nonumber\\
 &\ll X^\eta MN^2Q^{-1}\sum_{d\leq X^\varepsilon} \frac{1}{d^2} \sum_{\substack{d_1 > X^\varepsilon\\ d_1 \mid d^\infty}}\frac{1}{d_1} +X^\eta  MN \sum_{d\leq X^\varepsilon} \frac{1}{d} \sum_{\substack{d_1 > X^\varepsilon\\ d_1 \mid d^\infty}}\frac{1}{d_1}\nonumber \\
 & \ll MN^2 Q^{-1}  X^{\eta -\frac{\varepsilon}{2}} +  X^{1 + \eta -\frac{\varepsilon}{2}}. \label{d1>X}
  \end{align}
  Consider the conditions 
  \begin{equation}\label{d<Xd1<X}
  d<X^\varepsilon \text{ and } d_1 < X^\varepsilon,
  \end{equation}
  and the subsum $\widetilde {W} (Q)$ of $W(Q)$ where the variables $n_1$ and $n_2$ satisfy the condition \eqref{d<Xd1<X}.
  By \eqref{d>X} and \eqref{d1>X} we have the equality
  \begin{equation}\label{W=tildeW}
  W(Q) = \widetilde{ W}(Q)+ O \bigl(  MN^2 Q^{-1}  X^{\eta-\frac{\varepsilon}{2}}+  X^{1 + \eta}\bigr).
\end{equation}
 
  \subsection{Expansion in Fourier series} We apply  Lemma \ref{existenceofpsi}  to the last sum over $m$ in \eqref{defW} with $H$ defined in \eqref{defH}. This decomposes $\widetilde{W}(Q)$ into the sum
  \begin{equation}\label{decompoW}
  \widetilde{W}(Q)= \widetilde{W}^{\rm MT}(Q) + \widetilde{W}^{\rm Err1} (Q) +\widetilde{ W}^{\rm Err2} (Q),
  \end{equation}
  where each of the three terms corresponds to the contribution of each term on the right--hand side of \eqref{ineqpsi}.
  
  The easiest term is $\widetilde{ W}^{\rm Err2} (Q)$ since, by $|\beta_n| \leq \tau_k(n)$ and \eqref{universal},  it satisfies the inequality
  \begin{align}
  \widetilde{ W}^{\rm Err2} (Q) &\ll M^{-1}  \sum_{Q / 2 \leq q \leq 5 Q / 2}\ \underset{\substack{n_1, n_2 \sim N\\ n_1 \equiv n_2 \bmod q}}{\sum \sum} \tau_k (n_1)   \tau_k (n_2)\nonumber\\
 & \ll M^{-1} N^2 X^\eta
  \label{WErr2}
   \end{align}
  According to the restriction \eqref{n1=n2}, we see that the main term is
  \begin{equation}\label{tildeWMT}
 \widetilde{ W}^{\rm MT} (Q) =  \hat \psi (0) M   \sum_{ (q,a) =1}  \frac{\psi (q/Q)}{q}\sum_{(\delta , q) =1} \Bigl(
 \underset {\substack{n_1, n_2 \sim  N\\ n_1\equiv n_2\equiv \delta \bmod q}}{ \sum\ \ \sum} \beta_{n_1}\beta_{n_2}\Bigr),
  \end{equation}
  where the variables $n_1$ and $n_2$ satisfy the conditions \eqref{d<Xd1<X}. By a similar computation leading  to \eqref{d>X} and \eqref{d1>X}  we can drop  these conditions at the cost of the same error term. In other words  the equality \eqref{tildeWMT} can be written as
    \begin{equation}\label{W=W}
   \widetilde{ W}^{\rm MT} (Q) =W^{\rm MT} (Q) + O \bigl(  MN^2 Q^{-1} X^{\eta-\frac{\varepsilon}{2}} +  X^{1 + \eta}\bigr),
 \end{equation}
   where  $W^{\rm MT} (Q)$ is the new main term, which is defined
   by 
   \begin{equation}\label{defWMT}
 W^{\rm MT} (Q) =  \hat \psi (0) M   \sum_{  (q,a) =1}  \frac{\psi (q/Q)}{q}\sum_{(\delta , q) =1} \Bigl( \sum_{\substack{n \sim  N\\ n \equiv \delta \bmod q}} \beta_{n} \Bigr)^2.
   \end{equation}
   
  \subsection{Dealing with the main terms}
  We now gather  the main terms appearing in \eqref{UMT},  \eqref{V3},    \eqref{W=tildeW}, \eqref{decompoW},  \eqref{W=W},  and in \eqref{defWMT}.  The main term of $W(Q)-2 V(Q) +U(Q)$ is  
  \begin{multline*}
  W^{\rm MT} (Q) -2 V^{\rm MT} (Q) + U^{\rm MT} (Q)\\ =\hat \psi (0) M  \sum_{ (q,a)=1} \frac{\psi (q/Q)}{q} \sum_{(\delta, q) =1}
  \Bigl( \sum_{\substack{n\sim N \\ n\equiv \delta \bmod q}} \beta_n -\frac{1}{\varphi (q)} \sum_{\substack{n \sim N \\ (n,q) =1}} \beta_n
  \Bigr)^2.
  \end{multline*}
  Appealing to Lemma \ref{Ba-Da-Ha} we deduce that, for any $A$,  we have the equality
  \begin{equation}\label{W-2V+U}
   W^{\rm MT} (Q) -2 V^{\rm MT} (Q) + U^{\rm MT} (Q) = O \Bigl(  M\cdot  Q^{-1}\cdot  N^2  (\log 2N)^{-A}\Bigr)
   \end{equation}
   provided that
   \begin{equation}\label{Qleq}
   Q \leq N(\log 2N)^{-B},
   \end{equation}
   for some  $B =B(A).$

   \subsection{Preparation of the exponential sums} By the definition \eqref{decompoW}, we have the equality
   \begin{equation*}
  \widetilde{ W}^{\rm Err1} (Q) =  M\sum_{q} \frac {\psi (q/Q)}{q} \underset{\substack{n_1,  n_2 \sim N\\ n_1 \equiv n_2 \bmod q}}{\sum \sum} \beta_{n_1} \beta_{n_2}
   \sum_{0 < \vert h \vert \leq H} \hat \psi \Bigl( \frac{h}{q/M}\Bigr)\,e \Bigl( \frac{ah \overline{n_1}}{q}\Bigr),
   \end{equation*}
   where the variables $(n_1,n_2)$ are such  the associated $d$ and $d_1$ satisfy  \eqref{d<Xd1<X}.
   
   This implies that any pair $(n_1, n_2)$ satisfies $n_1-n_2 \not= 0$ and since  we have 
 $n_1\equiv n_2 \bmod q$  (see \eqref{n1=n2}) these integers cannot be near to each other, indeed they satisfy the inequality
 \begin{equation*}
 \vert n_1-n_2 \vert \geq  Q/2.
 \end{equation*}
 Since we have $(n_1n_2, q)=1$, we can equivalently  write the congruence $n_1-n_2 \equiv 0 \bmod q$ as
\begin{equation}\label{nu1-nu2}
 \nu_1-\nu_2 = d_1 \nu'_1-\nu_2= qr,
 \end{equation}
 and instead of summing over $q$, we will sum over $r$. Note that $1 \leq \vert r \vert \leq  R/d$, where 
 \begin{equation}\label{defR}
 R =2N Q^{-1}.
 \end{equation}
 In the summations,  the  pair of variables $(n_1,n_2)$ is replaced by  the quadruple $(d, d_1, \nu'_1, \nu_2)$  (see \eqref{decompn1n2}). The variables $d$ and $d_1$ are small, so we expect no substantial cancellations
 when summing over them. Hence  for some

 \begin{equation*}
 d,\,d_1\leq X^\varepsilon,\ d_1 \mid d^\infty, 
 \end{equation*} 
 we have the inequality
\begin{equation}\label{587}
   \widetilde{W}^{\rm Err1} (Q) \ll X^{2\varepsilon } MQ^{-1}  \bigl\vert \, \mathcal W \,  \bigr\vert,
  \end{equation}
  where $\mathcal W= \mathcal W(d,d_1) $ is the quadrilinear form in the four variables $r$, $\nu'_1$, $\nu_2$ and $h$ defined by
  \begin{multline}\label{defmathcalW}
 \mathcal W = \sum_{1\leq \vert r \vert \leq R/d} \underset{\substack{dd_1\nu'_1, d\nu_2 \sim N \\ d_1\nu'_1\equiv  \nu_2 \bmod r}}
  {\sum \sum}\beta_{dd_1 \nu'_1} \beta_{d\nu_2}
   \frac{ \psi \bigl( (d_1\nu'_1-\nu_2)/ (rQ)\bigr)}{ (d_1\nu'_1-\nu_2)/ (rQ)}\\
    \sum_{0 <\vert h \vert \leq H}  \hat \psi \Bigl( \frac{h}{(d_1\nu'_1-\nu_2)/(rM)}\Bigr)
  e(\cdot ),
  \end{multline}
  where $e(\cdot)$ is the oscillating factor
  $$
  e(\cdot) = e \Bigl( \frac{ah\, \overline{dd_1 \nu'_1}}{(d_1 \nu'_1- \nu_2)/r} \Bigr),
  $$
  and  where the variables satisfy the following divisibility conditions: 
    $$
  (d_1\nu'_1, \nu_2) =1,\ (\nu'_1, d) =1 \text{ and } (dd_1 \nu'_1 r, d_1\nu'_1 - \nu_2) =r.
  $$
  Using B\'ezout's reciprocity formula we transform the factor $e(\cdot)$ as follows: 
  \begin{equation*}
 \frac{ah\, \overline{dd_1 \nu'_1}}{(d_1 \nu'_1- \nu_2)/r} = -  ah\frac{ \overline {(d_1 \nu'_1- \nu_2)/r}}{dd_1 \nu'_1} +\frac{ahr}{dd_1 \nu'_1 
( d_1\nu'_1 - \nu_2) }  \bmod 1.
\end{equation*}
 Since $(dd_1, \nu'_1)= (r,\nu'_1) =1$ we can apply B\'ezout formula again, giving the equalities
\begin{align*}
ah \frac{ \overline {(d_1 \nu'_1- \nu_2)/r}}{dd_1 \nu'_1} &=ah \frac{ \overline{\nu'_1} \, \overline {(d_1 \nu'_1- \nu_2)/r}}{dd_1} +ah  \frac {\overline{dd_1}\, \overline {(d_1 \nu'_1- \nu_2)/r}}{\nu'_1}\bmod 1\\
& = ah \frac{ \overline{\nu'_1} \, \overline {(d_1 \nu'_1- \nu_2)/r}}{dd_1} -ah \frac{r \overline{dd_1\nu_2}}{\nu'_1} \bmod 1 
\end{align*}
The first term on the right--hand side of the above equality depends only on the congruences classes of  $a$, $h$, $r$, $\nu'_1$ and $\nu_2$ modulo $dd_1$. 
As a consequence of the above discussion, we see that there exists a coefficient $\xi=\xi (a,h, r ,\nu'_1, \nu_2)$ of modulus $1$, depending only on the congruence classes of $a$, $h$, $r$, $\nu'_1$ and $\nu_2$ modulo $dd_1$ such that we have the equality
$$
e(\cdot )= \xi\cdot e \Bigl( \frac{ahr}{dd_1 \nu'_1 
( d_1\nu'_1 - \nu_2) }\Bigr) \cdot e \Bigl(  \frac{ah r \overline{dd_1\nu_2} }{\nu'_1} \Bigr).
$$
Returning to \eqref{defmathcalW},  and fixing the congruences classes modulo $dd_1$ of the variables $h$,  $r$, $\nu'_1$ and $\nu_2$, we see that there exists
$$0 \leq a_1, a_2, a_3, a_4 < d d_1$$
such that 
$\mathcal W$
satisfies the inequality, 
\begin{multline}\label{W1}
\vert  \mathcal W \bigr\vert \leq X^{6\varepsilon} \\ \sum_{\substack{1\leq \vert r\vert \leq R/d \\ r \equiv a_1 \bmod{d d_1}}} 
\Bigl\vert 
\underset{\substack{dd_1\nu'_1, \, d\nu_2 \sim N \\ d_1\nu'_1\equiv  \nu_2 \bmod r \\ \nu_1' \equiv a_2 \bmod d d_1 \\ \nu_2 \equiv a_3 \bmod d d_1}}
  {\sum \sum}\beta_{dd_1 \nu'_1} \beta_{d\nu_2} \sum_{\substack{1\leq \vert h \vert \leq H \\ h \equiv a_4 \bmod d d_1}}
\Psi_r (h, \nu'_1, \nu_2)
 e \Bigl(  \frac{ah r \overline{dd_1\nu_2} }{\nu'_1} \Bigr)
\Bigr\vert,
\end{multline}
 where $\Psi_r$ is the differentiable function 
$$
\Psi_r (h, \nu'_1, \nu_2)=  \frac{ \psi \bigl( (d_1\nu'_1-\nu_2)/ (rQ)\bigr)}{ (d_1\nu'_1-\nu_2)/ (rQ)} \,  \hat \psi \Bigl( \frac{h}{(d_1\nu'_1-\nu_2)/(rM)}\Bigr)\, e \Bigl(\frac{ahr}{dd_1 \nu'_1 
( d_1\nu'_1 - \nu_2) } \Bigr),
$$


 In order to perform
the Abel summation over the variables $\nu'_1$, $\nu_2$ and $h$  (see for instance \cite[Lemme 5]{FoActaMath}) we must have information on the partial derivatives of the $\Psi_r$--function. Indeed  for $0\leq \epsilon_0,\, \epsilon_1, \epsilon_2 \leq 1$, we have the inequality   
\begin{equation}\label{bypart}
\frac{\partial^{\epsilon_0+ \epsilon_1+ \epsilon_2}}{\partial h^{\epsilon_0}\partial {\nu'_1}^{\epsilon_1} \partial\nu_2^{\epsilon_2}} \Psi_r (h, \nu'_1, \nu_2) \ll X^{50\varepsilon} \vert h\vert^{-\epsilon_0} \, {\nu'_1}^{-\epsilon_1} \, {\nu_2}^{-\epsilon_2} \bigl(N/(rQ)\bigr)^{ \epsilon_1 + \epsilon_2 },
\end{equation}
as a consequence of  the inequality $\vert d_1\nu'_1 -\nu_2 \vert \geq rQ/2$ (see\eqref{nu1-nu2}), of the definition of $H$ (see \eqref{defH}) and of the inequality $1 \leq |a| \leq X$.

Since $(d_1 \nu'_1 \nu_2,r)=1$  we detect the congruence $d_1 \nu'_1 \equiv \nu_2 \bmod r$ by    the $ \varphi (r)$ Dirichlet characters  $\chi$ modulo $r$. By \eqref{bypart} we eliminate the function $\Psi_r$ in the inequality \eqref{W1} which becomes 
\begin{multline}\label{W2}
\vert  \mathcal W \bigr\vert \leq X^{60\varepsilon} N^2Q^{-2}  \sum_{\substack{1\leq \vert r\vert \leq R/d \\ r \equiv a_1 \bmod d d_1}} \frac{1}{\varphi (r) \,  r^2} \sum_{\chi \bmod r}\\ 
\Bigl\vert 
\underset{\substack{dd_1\nu'_1\in \mathcal N_1\, d\nu_2 \in \mathcal N_2 \\ \nu_1' \equiv a_2 \bmod d d_1 \\ \nu_2 \equiv a_3 \bmod d d_1}} 
  {\sum \sum}\chi (d\nu'_1) \overline{\chi}(\nu_2) \beta_{dd_1 \nu'_1} \beta_{d\nu_2} \sum_{\substack{h\in \mathcal H \\ h \equiv a_4 \bmod d d_1}} 
  e \Bigl(  \frac{ah r \overline{dd_1\nu_2} }{\nu'_1} \Bigr)
\Bigr\vert,
\end{multline}
\noindent $\bullet$ where $\mathcal N_1 $ and $\mathcal N_2$ are two intervals included in $[N, 2N]$,

\noindent $\bullet$ and where $\mathcal H$ is the union of two intervals included in $[-H, -1]$ and $[1, H]$ respectively.

 Denote by $\mathcal W_1(r, \chi)$ the  inner sum over $\nu'_1$, $\nu_2$ and $h$ in \eqref{W2}.  Remark that the trivial bound for $\mathcal W_1 (r, \chi)$ is $O( X^\eta  H  N^2/(d^2d_1))$. We now can apply Lemma \ref{trilinear} to the sum $\mathcal W_1 (r, \chi)$, with the choice of parameters
 $$
 \vartheta \rightarrow  ar, \ A \rightarrow H, M\rightarrow  N \text{ and } N \rightarrow N .
 $$
 We obtain the bound
 $$
 \mathcal W_1 (r, \chi) \ll  H^\frac 12 N^\frac 12 N^\frac 12  X^{ \varepsilon +\eta}\Bigl( 1 + \frac{\vert a \vert \, \vert r\vert  H}{N^2}\Bigr) \Bigl( (HN^2)^{\frac{7}{20} + \varepsilon} N^\frac 14+ (HN^2)^{\frac 38 + \varepsilon} (HN)^\frac 18 \Bigr).
 $$
By the definition    \eqref{defR}, \eqref{defH}  and the inequality $1 \leq |a| \leq X$   we deduce the inequality
\begin{equation*} 
\mathcal W_1 (r, \chi) \ll X^{4\varepsilon  +\eta} \bigl( H^\frac{17}{20} N^\frac{39}{20} + H N^\frac{15}{8}\bigr),
\end{equation*} 
 and using \eqref{defH} we finally deduce
$$  
\mathcal W_1 (r, \chi) \ll X^{5\varepsilon +\eta}  \bigl( M^{-\frac{17}{20}} N^\frac{39}{20} Q^\frac{17}{20}  + M^{-1} N^\frac{15}{8}Q\bigr).
$$
 Returning to \eqref{W2},  summing over $\chi$ and  $r$  and inserting into \eqref{587} we obtain the bound
\begin{equation}\label{tildeWErr2<<}
  \widetilde{W}^{\rm Err1} (Q) \ll X^{67 \varepsilon +\eta } \bigl( M^\frac {3}{20} N^\frac{79}{20} Q^{-\frac{43}{20} } + N^\frac{31}{8 } Q^{-2}
  \bigr).
\end{equation}
 \subsection{Conclusion}  We have now all the elements to bound $ \Delta (\boldsymbol \alpha, \boldsymbol \beta, M, N, Q, a)$.
 By \eqref{CS}, \eqref{UMT}, \eqref{V5}, \eqref{W=tildeW}, \eqref{decompoW},  \eqref{WErr2}, \eqref{W=W} and \eqref{tildeWErr2<<} we  have 
the inequality
\begin{multline*}
 \Delta^2 \ll MQ\mathcal L^{k^2-1} \Bigl\{ \Bigl( W^{\rm MT} (Q) -2 V^{\rm MT} (Q)+ U^{\rm MT} (Q)\Bigr) + N^2Q^{-1} X^\eta \\
+ (M^{-1} N^2 +N ^\frac 52 Q^{-1})X^{2\varepsilon+\eta}
 \\+ \bigl( MN^2 Q^{-1} X^{\eta -\frac{\varepsilon}{2}}  +
 X^{1 + \eta} \bigl)  + X^{67 \varepsilon +\eta } \bigl( M^\frac {3}{20} N^\frac{79}{20} Q^{-\frac{43}{20} } + N^\frac{31}{8 } Q^{-2}
  \bigr)
\Bigr\}, 
\end{multline*}
which is shortened in (recall \eqref{universal})
\begin{multline*}
 \Delta^2 \ll MQ\mathcal L^{k^2-1} \Bigl\{  MN^2 Q^{-1} (\log 2N)^{-A} +  \\+   MN^2 Q^{-1} X^{\eta -\frac{\varepsilon}{2}}   
  + X^{67 \varepsilon +\eta } \bigl( M^\frac {3}{20} N^\frac{79}{20} Q^{-\frac{43}{20} } + N^\frac{31}{8 } Q^{-2}
  \bigr)
\Bigr\}, 
\end{multline*}
by  \eqref{W-2V+U} and  \eqref{Qleq} if one assumes  
\begin{equation}\label{10}
Q \leq N X^{-\varepsilon}.
\end{equation}
To finish the proof of Theorem \ref{thm:main}, it remains to find sufficient conditions 
   over $M$, $N$ and $Q$ to ensure the bound $\Delta^2 \ll M^2 N^2 \mathcal L^{-A}$. Choosing $\eta= \varepsilon /5$, we have to study  the following three inequalities hold
    \begin{equation}\label{system}
   \begin{cases} 
 MQ \cdot  MN^2 Q^{-1} X^{-\frac{\varepsilon}{4} }&\ll M^2N^2X^{-\frac{\varepsilon}{4}}, \\
 MQ \cdot M^\frac{3}{20} N^\frac{79}{20} Q^{-\frac{43}{20} }X^{68 \varepsilon}& \ll M^2N^2X^{-\frac{\varepsilon}{4}}, \\
  MQ \cdot  N^\frac{31}{8} Q^{-2} X^{68 \varepsilon} & \ll M^2N^2X^{-\frac{\varepsilon}{4}}. 
   \end{cases}
   \end{equation}
  The first inequality   is trivially satisfied. The second  inequality  of \eqref{system} is satisfied  as soon as 
  \begin{equation}\label{14}
     Q> N^\frac{56}{23} X^{-\frac {17 }{23} +65 \varepsilon}.
\end{equation}
 This inequality combined with  \eqref{10} implies that  $N< X^\frac{17}{33}$.  
The last condition   of \eqref{system} is   satisfied as soon as 
$$
Q > N^\frac{23}{8} X^{-1 +69 \varepsilon}. 
$$ 
We can   drop this condition since it is a consequence of \eqref{14} and of the inequality $N < X^\frac{17}{33}$. The proof of Theorem \ref{thm:main} is now complete.
   
   \section{Proof of Corollary \ref{appli}}  \label{proofcorollary}  Let $S(M,N)$ be the sum we are studying in this corollary. We use Dirichlet's hyperbola argument  to write 
  \begin{equation}\label{Diri}
  mn -1=qr,
  \end{equation}
and by symmetry we can impose  the condition $q< r$. This symmetry creates a factor $2$ unless $mn-1$ is a perfect square.  The contribution to $S(M,N)$  of the $(m,n)$ such that $mn-1$ is a square is bounded by $O( X^{\frac 12 +\eta})$ with $\eta >0$ arbitrary. This is a consequence of $|\beta_n| \leq \tau_k(n)$.

 The decomposition \eqref{Diri}, the  constraint  $q < r$   and the inequalities $X-1 \leq mn-1< 4X$   imply  that $q\leq 2 X^\frac 12$. In counterpart, if $q <X^\frac 12$ we are sure that $q <r$.  Thus we have the equality
 \begin{align}\label{END}
 S(M,N) &=
 2 \sum_{q \leq X^{1/2}} \underset{\substack{ m\sim M , n\sim N \\ mn\equiv 1 \bmod q}}{ \sum \sum} \alpha_m \beta_n  + 2 \underset{\substack{mn-1 =qr, q<r \\
 m\sim M, n\sim N \\ X^{1/2} < q  \leq 2 X^{1/2}} }{\sum\ \sum\  \sum\  \sum} 
  \alpha_m  \beta_n + O(X^{\frac 12 +\eta})\nonumber\\
 &= 2S_0 (M,N) +2 S_1 (M,N) + O(X^{\frac 12 +\eta}),
 \end{align} 
 by definition.
 A direct application of  Theorem \ref{thm:main}  with $Q= X^\frac 12$  gives the equality
 \begin{equation}\label{Eq1}
 S_0 (M,N)=    \sum_{q\leq X^{1/2}} \frac{1}{\varphi (q)} \underset{\substack{m\sim M, n\sim N\\ (mn,q) =1}}{\sum  \sum} \alpha_m \beta_n + O( X \mathcal L^{-C}),
 \end{equation}
 for any $C$. 
 
 For the second term $S_1 (M,N)$,  we must get rid of the constraint $q<r$. A technique among others is to precisely control the size of the variables $m$, $n$ and $q$.
If it is so, then  $r=(mn-1)/q$ is also controlled and one can check if it satisfies $r >q$. 
We introduce  the following factor of dissection:
 $$
 \Delta = 2^\frac{1}{[ \mathcal L^B]},
 $$
 where $B= B(A)$ is a parameter to be fixed later, and where $[y]$ is the largest integer $\leq y$.   If we denote by $L_0= [\mathcal L^B]$ we see that $\Delta^{L_0} =2$ and that $\Delta =1 + O (\mathcal L ^{-B})$. We denote by $M_0$, $N_0$ and $Q_0$ any numbers in the sets
 \begin{align*}
 \mathcal M_0&:=\{M, \Delta M, \Delta^2 M, \Delta^3 M,\cdots ,\Delta^{L_0-1} M\}\\
 \mathcal N_0&:= \{N, \Delta N, \Delta^2 N, \Delta^3 N,\cdots , \Delta^{L_0-1} N\}\\
  \mathcal Q_0 & :=\{X^\frac 12, \Delta X^\frac 12, \Delta^2 X^\frac 12, \Delta^3 X^\frac 12,\cdots, \Delta^{L_0-1} X^\frac 12\},
 \end{align*} respectively.
 We split $S_1 (M,N)$ into 
 \begin{equation}\label{decompo1}
 S_1 (M,N) =\sum_{M_0 \in \mathcal M_0} \  \sum_{N_0 \in \mathcal N_0} \ \sum_{Q_0 \in \mathcal Q_0} S_1(M_0,N_0,Q_0),
 \end{equation}
   where  $S_1(M_0,N_0,Q_0)$ is defined by
 $$
 S_1(M_0,N_0,Q_0)=\sum_{ q\simeq Q_0}\  \underset{\substack{  m\simeq M_0,  n  \simeq N_0 \\ mn\equiv 1 \bmod q}}{ \sum \sum}\alpha_m \beta_n.
 $$
  \noindent $\bullet$ where the notation $y\simeq Y_0$ means that the integer $y$ satisfies the inequalities  $Y_0 \leq y < \Delta Y_0$,
 \vskip .2cm
  \noindent $\bullet$ where the variables $m$, $n$ and $q$ satisfy the extra condition 
 \begin{equation}\label{r>q}
mn-1>q^2.
\end{equation}
Note that the decomposition \eqref{decompo1} contains 
\begin{equation}\label{O(LB)}
O( \mathcal L^{3B}),
\end{equation}
terms.

Since $mn-1 \geq  M_0N_0 -1  $ and  $q^2  <Q_0^2 \Delta^2$ in each sum $S_1(M_0,N_0, Q_0)$, we can drop the condition 
\eqref{r>q} in the definition of this sum as soon as we have
\begin{equation} \label{M0N0<Q0}
M_0N_0-1   >Q_0^2 \Delta^2.
\end{equation}
When   \eqref{M0N0<Q0} is satisfied, the variables $m$, $n$ and $q$ are independent and a direct application of Theorem \ref{thm:main} gives for each sum $S_1 (M_0, N_0, Q_0) $, the equality
\begin{equation}\label{nunu}
S_1 (M_0,N_0, Q_0)= \sum_{ q\simeq Q_0} \frac{1}{\varphi (q)}  \underset{\substack{  m\simeq M_0,  n  \simeq N_0\\ (mn,q)=1}}{ \sum  \sum}\alpha_m \beta_n +O_C (X \mathcal  L^{-C}),
\end{equation}
where  $C$ is arbitrary.  

It remains to consider the case where  \eqref{M0N0<Q0} is not satisfied, which means that  $(M_0, N_0, Q_0) \in \mathcal E_0$ where 
\begin{equation}\label{inverse}
\mathcal E_0 := \bigl\{ (M_0, N_0, Q_0)\, ;\ M_0N_0-1  \leq  Q_0^2 \Delta^2\bigr\}.
\end{equation}
We now  show that the variable $n$ considered in such a  $S_1 (M_0, N_0, Q_0)$  varies in a rather short interval. More precisely, since   $M_0 \Delta>m, N_0 \Delta  >n$ and $Q_0 < q\ $ we deduce from the   definition 
\eqref{inverse} that $q^2  \geq  mn  \Delta^{-4} -\Delta^{-2}$ which implies the inequality $q \geq  (mn)^\frac 12 \Delta^{-2}  -1$.  Combining with  \eqref{r>q}, we get  the inequality
$$
(mn)^\frac 12 \Delta^{-2} -1 < q < (mn)^\frac 12  
$$
which implies 
 $$
 (q^2/m)  < n< ((q+1)^2/m)\Delta^4.
 $$
Using the inequality
 $$
 X^{1/2} \leq q \leq 2 X^{1/2} \ll  (Q^2/M)(\Delta^4-1)X^{-\frac{\delta}{2}},
 $$ 
 and $|\beta_n| \leq \tau_k(n)$ we apply  Lemma \ref{dkinarith} to see that
 \begin{align*}
&\underset{(M_0,N_0,Q_0) \in \mathcal E_0}{\sum\ \sum\ \sum} S_1( M_0, N_0, Q_0)\\ 
&\ll    \sum_{m \sim M} \tau_k (m) \sum_{\substack{q \sim X^{1/2}\\ (q,m)=1}} \sum_{(q^2/m)  < n< ((q+1)^2/m)\Delta^4}\tau_k (n) \\
&\ll (\Delta^4 -1) \mathcal L^{k-1}  \sum_{m\sim M} \tau_k (m) 
\sum_{q \sim X^{1/2} }\frac{1}{\varphi (q)}\cdot \frac{q^2}{m}\\
&\ll \mathcal L^{2k-2-B} X.
\end{align*}
 
 Actually, by introducing a main term back, which is less than the error term, we can also write this bound as an equality
\begin{multline}
 \underset{(M_0,N_0,Q_0) \in \mathcal E_0}{\sum\ \sum\ \sum} S_1( M_0, N_0, Q_0)
 \\
 = \underset{(M_0,N_0,Q_0) \in \mathcal E_0}{\sum\ \sum\ \sum} \sum_{ q\simeq Q_0}\ \frac{1}{\varphi (q)} \underset{\substack{ m\simeq M_0,  n  \simeq N_0\\ (mn,q)=1}}{ \sum \sum}\alpha_m \beta_n
 +O(   \mathcal L^{2k-2-B} ), \label{Eq2}
 \end{multline}
  where the variables $(m,n,q)$ continue to satisfy \eqref{r>q}.
 
 Gathering \eqref{END}, \eqref{Eq1},  \eqref{decompo1},    \eqref{O(LB)}, \eqref{nunu},  \eqref{Eq2}
 we obtain 
\begin{multline*}
 S(M,N)=  2 \sum_{q\leq X^{1/2}} \frac{1}{\varphi (q)} \underset{\substack{m\sim M, n\sim N\\ (mn,q) =1}}{\sum \sum} \alpha_m \beta_n\\+2\sum_{M_0 \in \mathcal M_0} \  \sum_{N_0 \in \mathcal N_0} \ \sum_{Q_0 \in \mathcal Q_0} \sum_{ q\simeq Q_0}\ \frac{1}{\varphi (q)} \underset{\substack{  m\simeq M_0,  n  \simeq N_0\\ (mn,q)=1}}{ \sum \sum}\alpha_m \beta_n \\
 +O( \mathcal L^{3B-C} X) +O(   \mathcal L^{2k-2-B} X) + O(X^{\frac 12 +\eta}),
\end{multline*}
   where the variables $(m,n,q)$ continue to satisfy \eqref{r>q}. 
   Putting the different  summations  back together, we complete the proof of Corollary \ref{appli} by choosing $B$ and $C$ in order to satisfy the equalities  $-A= 3B-C = 2k-2-B$.

\bibliographystyle{plain} \bibliography{FouvryRadzi4.bib}

 \end{document}